\documentclass[a4paper, 12 pt, oneside, reqno]{amsart}
\usepackage{amsfonts, amssymb, amsmath, eucal, amsthm}
\usepackage[all]{xy}
\usepackage{verbatim}
\usepackage{tikz}

\newtheorem{lemma}{Lemma}
\newtheorem{theorem}[lemma]{Theorem}
\newtheorem*{theorem*}{Theorem}
\newtheorem*{corollary*}{Corollary}
\newtheorem{proposition}[lemma]{Proposition}
\newtheorem{corollary}[lemma]{Corollary}
\newtheorem{claim}{Claim}
\newtheorem*{main}{Main Theorem}

\theoremstyle{definition}
\newtheorem{definition}[lemma]{Definition}

\newtheorem{example}[lemma]{Example}
\newtheorem{remark}[lemma]{Remark}

\usepackage[vmargin=3.2cm, hmargin=3.4cm]{geometry}

\newcommand{\C}{\mathcal{C}}
\DeclareMathOperator{\sd}{sd}
\newcommand{\U}{\mathcal{U}}
\newcommand{\In}{\mathrm{In}}
\newcommand{\lk}{\mathrm{lk}}
\newcommand{\ord}{\mathrm{ord}}
\newcommand{\D}{\mathcal{D}}

\title{Homological stability for families of Coxeter groups}
\author{Richard Hepworth}
\address{Institute of Mathematics\\
University of Aberdeen\\
Aberdeen AB24 3UE\\
United Kingdom}
\email{r.hepworth@abdn.ac.uk}


\begin{document}

\begin{abstract}
We prove that certain families of Coxeter groups and inclusions
$W_1\hookrightarrow W_2\hookrightarrow\cdots$
satisfy homological stability, meaning that 
in each degree the homology $H_\ast(BW_n)$
is eventually independent of $n$.
This gives a uniform treatment of homological stability
for the families of Coxeter groups of type $A_n$, $B_n$ and $D_n$,
recovering existing results in the first two cases, and giving
a new result in the third.
The key step in our proof is to show that a certain
simplicial complex with $W_n$-action is highly connected.
To do this we show that the barycentric subdivision
is an instance of the `basic construction',
and then use Davis's description of the 
basic construction as an increasing union of chambers
to deduce the required connectivity. 
\end{abstract}

\maketitle

\section{Introduction}

A family of groups 
$G_1\hookrightarrow G_2\hookrightarrow G_3\hookrightarrow\cdots$
is said to satisfy \emph{homological stability} if the induced
maps $H_i(BG_{n-1})\to H_i(BG_n)$ are isomorphisms
when $n$ is sufficiently large relative to $i$.  Homological stability
is known for many families of groups, including symmetric
groups~\cite{Nakaoka},
general linear groups~\cite{Quillen},
mapping class groups of surfaces~\cite{Harer} 
and $3$-manifolds~\cite{HatcherWahl},
diffeomorphism groups of highly connected 
manifolds~\cite{GalatiusRandalWilliams},
and automorphism groups of free groups~\cite{HatcherAutFn},
\cite{HatcherVogtmann}.
\emph{Coxeter groups} are abstract reflection groups, appearing
in many areas of mathematics, such as root systems and
Lie theory, geometric group theory, 
and combinatorics.
See~\cite{Bourbaki}, \cite{Davis}, 
and \cite{BjornerBrenti} for introductions to Coxeter groups
from each of these three viewpoints.
In this paper we will show that homological 
stability holds for certain families of Coxeter groups.

Recall that a \emph{Coxeter matrix} on a set $S$ 
is an $S\times S$ symmetric
matrix $M$, with values in $\mathbb{N}\cup\{\infty\}$,
satisfying $m_{st}=1$
if $s=t$ and $m_{st}\geqslant 2$ otherwise.  The corresponding
\emph{Coxeter group} is the group generated by the elements of
$S$, subject to the relations $(st)^{m_{st}}=e$ for $s,t\in S$.
(When $m_{st}=\infty$ no relation is imposed.)
It is common to represent a Coxeter matrix
by the equivalent \emph{Coxeter diagram}.
This is the graph with vertices $S$ and edges $\{s,t\}$
for $m_{st}\geqslant 3$. The edge $\{s,t\}$ is labelled
$m_{st}$ if $m_{st}\geqslant 4$.

Now consider a sequence of finite Coxeter diagrams 
$(\Gamma_n)_{n\geqslant 1}$ of the form
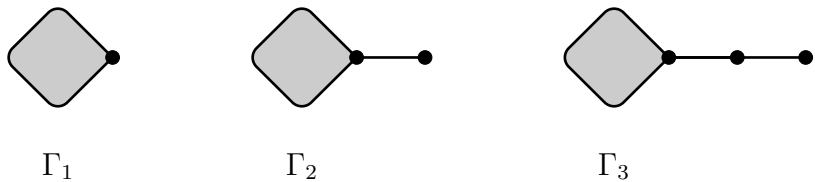
\[
\begin{tikzpicture}[scale=0.18, baseline=0]
	\draw[line width=1, fill = white!80!black, rounded corners=5 pt]
		(5,0) --  (1,4) -- (-3,0) -- (1,-4) -- (5,0);
	\draw[fill= black] (5,0) circle (0.5);
	\node at (1,-8) {$\Gamma_1$};
\end{tikzpicture}
\qquad\qquad
\begin{tikzpicture}[scale=0.18, baseline=0]
	\draw[line width=1, fill = white!80!black, rounded corners=5 pt]
		(5,0) --  (1,4) -- (-3,0) -- (1,-4) -- (5,0);
	\draw[fill= black] (5,0) circle (0.5);
	\draw[line width=1] (5,0) -- (10,0);
	\draw[fill= black] (10,0) circle (0.5);
	\node at (1,-8) {$\Gamma_2$};
\end{tikzpicture}
\qquad\qquad
\begin{tikzpicture}[scale=0.18, baseline=0]
	\draw[line width=1, fill = white!80!black, rounded corners=5 pt]
		(5,0) --  (1,4) -- (-3,0) -- (1,-4) -- (5,0);
	\draw[fill= black] (5,0) circle (0.5);
	\draw[line width=1] (5,0) -- (10,0);
	\draw[fill= black] (10,0) circle (0.5);
	\draw[line width=1] (5,0) -- (15,0);
	\draw[fill= black] (15,0) circle (0.5);
	\node at (1,-8) {$\Gamma_3$};
\end{tikzpicture}
\]
where every diagram has a preferred vertex,
and each diagram is obtained from its predecessor
by attaching a new preferred vertex 
to the old one by an unlabelled edge.
Writing $W_n$ for the Coxeter group determined by $\Gamma_n$,
the inclusion $\Gamma_{n-1}\hookrightarrow\Gamma_n$ 
induces an inclusion $W_{n-1} \hookrightarrow W_{n}$,
and our main result states that the family
\[
	W_1\hookrightarrow W_2\hookrightarrow W_3\hookrightarrow
	W_4 \hookrightarrow\cdots
\]
satisfies homological stability.

\begin{main}
The map
$
	H_\ast(BW_{n-1})\to H_\ast(BW_n)
$
is an isomorphism in degrees $2\ast\leqslant n$.
Here homology is taken with arbitrary constant
coefficients.
\end{main}

Observe that while the diagrams $\Gamma_n$ are assumed
to be finite, it is not necessary for the groups
$W_n$ to be finite.

\subsection*{Homological stability for Coxeter groups
of type $A_n$, $B_n$ and $D_n$}
The main theorem gives a uniform treatment of homological stability
for the families of Coxeter groups of type $A_n$, 
$B_n$ and $D_n$.
Recall that these are the Coxeter groups
corresponding to the following diagrams,
in which $n$ always denotes the total number of vertices.
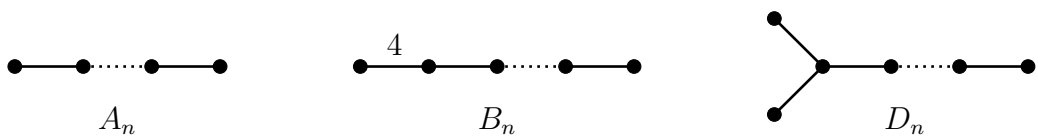
\[
\begin{tikzpicture}[scale=0.18, baseline=-3]
	\draw[fill= black] (5,0) circle (0.5);
	\draw[fill= black] (10,0) circle (0.5);
	\draw[fill= black] (15,0) circle (0.5);
	\draw[fill= black] (20,0) circle (0.5);
	\draw[line width=1] (5,0)  -- (10,0);
	\draw[line width=1, dotted] (10,0) -- (15,0);
	\draw[line width=1] (15,0) -- (20,0);
	\node at (12.5,-4) {$A_n$};
\end{tikzpicture}
\qquad\qquad
\begin{tikzpicture}[scale=0.18, baseline=-3]
	\draw[fill= black] (0,0) circle (0.5);
	\draw[line width=1] (0,0) -- node[above] {$4$} (5,0);
	\draw[fill= black] (5,0) circle (0.5);
	\draw[fill= black] (10,0) circle (0.5);
	\draw[fill= black] (15,0) circle (0.5);
	\draw[fill= black] (20,0) circle (0.5);
	\draw[line width=1] (5,0) -- (10,0);
	\draw[line width=1, dotted] (10,0) -- (15,0);
	\draw[line width=1] (15,0) -- (20,0);
	\node at (10,-4) {$B_n$};
\end{tikzpicture}
\qquad\qquad
\begin{tikzpicture}[scale=0.18, baseline=-3]
	\draw[fill=black] (5,0)+(135:5) circle (0.5);
	\draw[line width=1] (5,0)+(135:5)-- (5,0);
	\draw[line width=1] (5,0)+(225:5) -- (5,0);
	\draw[fill=black] (5,0)+(225:5) circle (0.5);
	\draw[fill= black] (5,0) circle (0.5);
	\draw[fill= black] (10,0) circle (0.5);
	\draw[fill= black] (15,0) circle (0.5);
	\draw[fill= black] (20,0) circle (0.5);
	\draw[line width=1] (5,0) -- (10,0);
	\draw[line width=1, dotted] (10,0) -- (15,0);
	\draw[line width=1] (15,0) -- (20,0);
	\node at (11,-4) {$D_n$};
\end{tikzpicture}
\]
These families have an important place in the theory of Coxeter
groups, since the classification of finite Coxeter groups
states that a finite irreducible Coxeter group has type $A_n$, $B_n$
or $D_n$, or is dihedral, or is one of six exceptional examples.
(See Appendix~C of~\cite{Davis}.)
The sequences $(A_n)_{n\geqslant 1}$,
$(B_{n+1})_{n\geqslant 1}$ and $(D_{n+2})_{n\geqslant 1}$
all have the form $(\Gamma_n)_{n\geqslant 1}$ described above,
with the rightmost vertex taken as the preferred vertex, and 
therefore we may apply the main theorem to each one.

For the sequence of diagrams $(A_n)_{n\geqslant 1}$,
the corresponding sequence of Coxeter groups is
\[
	\Sigma_2\hookrightarrow
	\Sigma_3\hookrightarrow
	\Sigma_4\hookrightarrow
	\Sigma_5\hookrightarrow\cdots
\]
where $\Sigma_{n}$ is the symmetric group on $n$ letters
and the inclusions are given by extending permutations 
by the identity.
Applying the main theorem, we recover the following classical result.

\begin{corollary*}[Nakaoka~\cite{Nakaoka}]
The map $H_\ast(B\Sigma_n)\to H_\ast(B\Sigma_{n+1})$
is an isomorphism in degrees $2\ast\leqslant n$.
\end{corollary*}

For the sequence of diagrams $(B_{n+1})_{n\geqslant 1}$,
the corresponding sequence of Coxeter groups 
\[
	C_2\wr\Sigma_2
	\hookrightarrow
	C_2\wr\Sigma_3
	\hookrightarrow
	C_2\wr\Sigma_4
	\hookrightarrow
	C_2\wr\Sigma_5
	\hookrightarrow
	\cdots	
\]
consists of the wreath products of the symmetric groups
with the group $C_2$ of order $2$,
and the inclusions are again given by extending permutations
by the identity.
Applying the main theorem gives the following 
special case of Hatcher and Wahl's result on homological
stability for wreath products.
(See Proposition~1.6 of~\cite{HatcherWahl}
and the discussion that follows it.)
It also follows from Randal-Williams's result on homological
stability for unordered configuration spaces.
(See Theorem~A of~\cite{RW} 
with $M=\mathbb{R}^\infty$ and $X=BC_2$.)

\begin{corollary*}[Hatcher-Wahl~\cite{HatcherWahl}]
The map $H_\ast(B(C_2\wr\Sigma_n))\to H_\ast(B(C_2\wr\Sigma_{n+1}))$
is an isomorphism in degrees $2\ast\leqslant n$.
\end{corollary*}

For the sequence of diagrams 
$(D_{n+2})_{n\geqslant 1}$, the corresponding sequence
of Coxeter groups is
\[
	H_3\hookrightarrow H_4
	\hookrightarrow H_5 \hookrightarrow H_6\hookrightarrow\cdots
\]
where $H_n$ denotes the kernel of the homomorphism
$C_2\wr\Sigma_n\to C_2$ that takes the sum of the $C_2$-components.
The main theorem gives the following result, which we believe to be new.

\begin{corollary*}
Let $H_n$ denote the Coxeter group of type $D_{n}$.
Then the inclusion $H_{n+1}\hookrightarrow H_{n+2}$
induces an isomorphism
$H_\ast(BH_{n+1})\to H_\ast(BH_{n+2})$ 
in degrees $2\ast\leqslant n$.
\end{corollary*}

The concrete descriptions of the groups of type $A_n$,
$B_n$ and $D_n$ that we used here can all be found
in section~6.7 of~\cite{Davis}.

\subsection*{Families of hyperbolic Coxeter groups}
The main theorem applies to interesting
families besides those of type $A_n$, $B_n$ and $D_n$
already considered. For example, if we fix an integer $m\geqslant 7$,
then the main theorem shows that homological stability
holds for the family of Coxeter groups
associated to the sequence of diagrams
$(\Gamma_n)_{n\geqslant 1}$
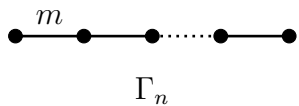
\[
\begin{tikzpicture}[scale=0.18, baseline=-3]
	\draw[fill= black] (0,0) circle (0.5);
	\draw[line width=1] (0,0) -- node[above] {$m$} (5,0);
	\draw[fill= black] (5,0) circle (0.5);
	\draw[fill= black] (10,0) circle (0.5);
	\draw[fill= black] (15,0) circle (0.5);
	\draw[fill= black] (20,0) circle (0.5);
	\draw[line width=1] (5,0) -- (10,0);
	\draw[line width=1, dotted] (10,0) -- (15,0);
	\draw[line width=1] (15,0) -- (20,0);
	\node at (10,-4) {$\Gamma_n$};
\end{tikzpicture}
\]
in which $\Gamma_n$ has a total of $(n+1)$ vertices,
the rightmost one preferred.
This family has the feature that the first group
is finite, while the rest are all infinite hyperbolic.
(Hyperbolicity is verified using Moussong's condition.
See Corollary~12.6.3 of~\cite{Davis}.)
It is not difficult to create other sequences
of infinite hyperbolic groups to which the main theorem applies.

On the other hand, these examples show that
it is possible for the first group in one of our families to be finite
while the rest are infinite.  Similarly, it is possible for 
the first group to be infinite hyperbolic while the rest are
not, for example if $\Gamma_1 = 
\begin{tikzpicture}[scale=0.1, baseline=-1]
	\draw[line width=1] (0,0) -- (5,0);
	\draw[line width=1] (5,0) -- (10,0);
	\draw[fill= black] (0,0) circle (0.5);
	\draw[fill= black] (5,0) circle (0.5);
	\draw[fill= black] (10,0) circle (0.5);
	\node at (2.5,1.5) {$\scriptstyle 6$} ;
	\node at (7.5,1.5) {$\scriptstyle 6$} ;
\end{tikzpicture}$
with any preferred vertex.  (The claims
about hyperbolicity again follow from Moussong's condition.)

\subsection*{Overview of the proof}
The proof of the main theorem proceeds as follows.

First, we construct  a simplicial complex 
$\C^n$ with an action of $W_n$, and prove 
that it is  {weakly Cohen-Macaulay of dimension $n$}, 
meaning that it is $(n-1)$-connected in a certain `homogeneous' way. 

Second, we form the semisimplicial set $\D^n$
whose simplices are simplices of $\C^n$
with an ordering of their vertices.
We show that $W_n$ acts transitively on the simplices in each dimension,
with stabilisers the subgroups $W_m\subset W_n$ for $m<n$.
From the weakly Cohen-Macaulay property of $\C^n$ we deduce that
the geometric realisation $\|\D^n\|$ is $(n-1)$-connected.

Third, we study the Borel construction $EW_n\times_{W_n}\|\D^n\|$.
From properties of $\D^n$ we deduce that the 
homology of the Borel construction matches
that of $BW_n$ in a range of degrees,
and that it can be computed by a spectral sequence
whose $E^1$-term consists of the homology groups 
$H_\ast(BW_m)$ for $m<n$.  An argument involving 
this spectral sequence completes the proof.

Let us explain in more detail how we
show that $\C^n$ is $(n-1)$-connected, since
this is by far the hardest step in the proof.
We make use of the `basic construction', a technique from the
topology of Coxeter groups (see chapters 5 and 8 of Davis's 
book~\cite{Davis}).  The basic construction
takes a Coxeter system $(W,S)$ and a `mirrored space' $X$ over
$S$, and produces a topological space $\U(W,X)$ with $W$-action.  
Now $\U(W,X)$ can be expressed as an `increasing union
of chambers', i.e.~copies of $X$, each copy of $X$ being attached
to the preceding ones in a controlled way.  
To prove that $\C^n$ is $(n-1)$-connected,
we show that the realisation of the barycentric 
subdivision of $\C^n$ has the form
$\U(W_n,|\Delta|)$, where $|\Delta|$ is a topological $n$-simplex.
By studying the attachments in the increasing union of chambers, 
we are able to deduce the required connectivity.

\subsection*{Outline of the paper}
In section~\ref{notation} we establish some notation,
and in section~\ref{preliminaries} we prove some 
elementary algebraic facts about
the groups $W_n$.
In section~\ref{Cn:section} we define
the simplicial complex $\C^n$. 
Then we examine $\C^n$ in detail: 
in section~\ref{links:section}
we study the links of its simplices, in section~\ref{action:section}
we study its action by $W_n$, and in section~\ref{basic:section}
we show that it is $(n-1)$-connected.
Then in section~\ref{ordered:section} we study the
semisimplicial set of ordered simplices in $\C^n$.
The proof of the main theorem is completed in 
section~\ref{completion:section}.

\subsection*{Acknowledgments}
My thanks to Jarek K\c{e}dra, Ian Leary
and Oscar Randal-Williams for helpful
conversations as this work was being carried out.

\section{Notation}
\label{notation}

\begin{definition}
Let $(\Gamma_n)_{n\geqslant 1}$ be a sequence
of the kind described in the introduction. 
We extend
this sequence to the left by two terms as follows.
Define $\Gamma_0$ to be the diagram obtained from $\Gamma_1$ 
by deleting the preferred vertex. And define
$\Gamma_{-1}$ to be the diagram obtained from $\Gamma_1$
by deleting the preferred vertex and all vertices
that shared an edge with it.
\end{definition}

\begin{example}[Coxeter groups of type $A_n$, $B_n$ and $D_n$]
For the sequence $(A_n)_{n\geqslant 1}$, the diagrams $A_0$ and $A_{-1}$
are both empty.  For $(B_{n+1})_{n\geqslant 1}$, the diagram 
$B_{0+1}$ consists of a single vertex and $B_{-1+1}$ is empty.
And for $(D_{n+2})_{n\geqslant 1}$, the diagram $D_{0+2}$ consists
of two vertices with no edge and $D_{-1+2}$ is empty.
\end{example}

\begin{definition}
Let $(\Gamma_n)_{n\geqslant 1}$ be a sequence
of the kind described in the introduction, and
let $(\Gamma_n)_{n\geqslant -1}$ be the extension
described above.  Then for $n\geqslant -1$ we define
$S_n$ to be the set of vertices of $\Gamma_n$, and
we define $W_n$ to be the Coxeter group associated
to $\Gamma_n$. Thus $(W_n,S_n)$ is a Coxeter system.
\end{definition}

\begin{definition}
For $n\geqslant 1$ we define $s_n\in S_n$ to be
the preferred vertex of $\Gamma_n$.
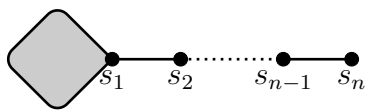
\[
\begin{tikzpicture}[scale=0.18, baseline=0]
	\draw[line width=1, fill = white!80!black, rounded corners=5 pt]
		(5,0) --  (1,4) -- (-3,0) -- (1,-4) -- (5,0);
	\draw[fill= black] (5,0) circle (0.5) node[below] {$s_1$};
	\draw[line width=1] (5,0) -- (10,0);
	\draw[fill= black] (10,0) circle (0.5) node[below] {$s_2$};
	\draw[line width=1, dotted] (10,0) -- (17.5,0);
	\draw[fill= black] (17.5,0) circle (0.5)node[below] {$s_{n-1}$};
	\draw[line width=1] (17.5,0) -- (22.5,0);
	\draw[fill= black] (22.5,0) circle (0.5)node[below] {$s_{n}$};
\end{tikzpicture}
\]
Thus $S_n=S_0\cup\{s_1,\ldots,s_n\}$, and the parabolic
subgroup of $W_n$ generated by $s_1,\ldots,s_n$
is an isomorphic copy of $\Sigma_{n+1}$, with
$s_i$ acting as the transposition of $i$ and $(i+1)$.
\end{definition}

\section{Algebraic preliminaries}
\label{preliminaries}
From this point onwards, unless stated otherwise
we fix the sequence $(\Gamma_n)_{n\geqslant 1}$
and the integer $n\geqslant 1$.

In several places we will consider the symbol
$s_i\cdots s_n W_{n-1}$ for $i$ in the range
$1\leqslant i\leqslant n+1$.  In the
case $i=n+1$ we take it to mean $W_{n-1}$.

\begin{proposition}\label{permutations}
Let $i$ lie in the range $1\leqslant i\leqslant n$.
Then left multiplication by the element $s_i$
fixes the set
\[
	\{
		s_1\cdots s_nW_{n-1},\ 
		s_2\cdots s_nW_{n-1},\ 
		\ldots,\ 
		s_n W_{n-1},\ 
		W_{n-1}
	\}.
\]
It acts on the set
by transposing $s_i\cdots s_nW_{n-1}$
and $s_{i+1}\cdots s_nW_{n-1}$, and
fixing the remaining elements.
\end{proposition}

\begin{proof}
The identities
\begin{align*}
	s_i(s_j\cdots s_n) &= (s_j\cdots s_n)s_i
	\text{ for }i<j-1
	\\
	s_i(s_{i+1}\cdots s_n) &= s_i\cdots s_n
	\\
	s_{i}(s_{i}\cdots s_n)&=s_{i+1}\cdots s_n
	\\
	s_i(s_j\cdots s_n) &= (s_j\cdots s_n)s_{i-1}
	\text{ for }i>j
\end{align*}
are simple to verify, and the claim follows
immediately.
\end{proof}

\begin{proposition}\label{intersection}
$
	W_{i-1}\cap (s_i\cdots s_n W_{n-1}s_n\cdots s_i)
	=
	W_{i-2}
$
for $1\leqslant i\leqslant n$.
\end{proposition}

\begin{proof}
The word $s_i\cdots s_n$ is $(W_{i-1},W_{n-1})$-reduced,
meaning that it is reduced and has no reduced expression 
beginning with a generator of $W_{i-1}$ or ending with a 
generator of $W_{n-1}$.  A result of Kilmoyer, Solomon and Tits
(see Lemma~2 of~\cite{Solomon} and the remarks that precede it)
then shows that $W_{i-1}\cap s_i\cdots s_n W_{n-1}s_n\cdots s_i$
is the sugroup generated by 
$T=S_{i-1}\cap (s_i\cdots s_n S_{n-1}s_n\cdots s_i)$.
So it will be enough to show that $T=S_{i-2}$.
It is immediate that $S_{i-2}\subset T$,
so suppose that $t\in T\setminus S_{i-2}$.
Thus
	$t\in S_{i-1}\setminus S_{i-2}$
and	
	$s_n\cdots s_i t s_i\cdots s_n\in S_{n-1}$.
By the first condition we have $m(s_i,t)\geqslant 3$.
By the second condition the word 
$s_n\cdots s_i t s_i\cdots s_n$ is not reduced,
so that we must be able to apply an M-move to it 
(see section~3.4 of~\cite{Davis}), 
and this is only possible if 
$m(s_i,t)$ is exactly $3$.  But in this case 
$s_n\cdots s_i t s_i\cdots s_n$ is already reduced,
contradicting the first condition.
\end{proof}

\begin{proposition}
\label{tuple-stabilizer}
Let $i$ lie in the range $1\leqslant i\leqslant n$.
If $\sigma,\tau\in W_n$ satisfy
\[\sigma s_j\ldots s_n W_{n-1}
=
\tau s_j\ldots s_nW_{n-1}\ \text{ for }
j=i,\ldots,n+1
\]
then $\sigma^{-1}\tau\in W_{i-2}$.
\end{proposition}

\begin{proof}
The proposition is equivalent to the claim that
\[
	W_{n-1}
	\cap
	(s_n W_{n-1}s_n)
	\cap
	\cdots
	\cap
	(s_i\cdots s_n W_{n-1}s_n\cdots s_i)
	=
	W_{i-2},
\]
which is proved by
downward induction on $i$.
The initial case $i=n+1$ is immediate,
and the induction step follows from
Proposition~\ref{intersection}.
\end{proof}

\begin{proposition}\label{distinct}
For $c\in W_n$ the cosets
\[
	c(s_1\cdots s_n)W_{n-1},\ 
	\ldots,\ 
	cs_n W_{n-1},\ 
	cW_{n-1}
\]
are pairwise distinct.
\end{proposition}

\begin{proof}
If $cs_j\cdots s_nW_{n-1}=cs_k\cdots s_nW_{n-1}$
with $j<k$, then $(s_n\cdots s_j)(s_k\cdots s_n)\in W_{n-1}$.
But 
$
	(s_n\cdots s_j)(s_k\cdots s_n)
	=
	(s_{k-1}\cdots s_n\cdots s_{k-1})(s_{k-2}\cdots s_j),
$
implying that $s_n\in W_{n-1}$, which is a contradiction.
\end{proof}

\section{The simplicial complex $\C^n$}
\label{Cn:section}
Now we introduce the simplicial complex that will be central
to our proof of the main theorem, and we prove that it is weakly
Cohen-Macaulay of dimension~$n$.  The proof relies on propositions
that will be established in the following three sections.

\begin{definition}[The simplicial complex $\C^n$]
\label{Cndefinition}
Given $n\geqslant 0$ we let $\C^n$ denote the $n$-dimensional
abstract simplicial complex with vertex set $W_n/W_{n-1}$
and with $k$-simplices given by the subsets
\[
	C=\{
		c (s_{n-k+1}\cdots s_n) W_{n-1},\ 
		\ldots,\ 
		c s_nW_{n-1},\ 
		c W_{n-1}	
	\}
\]
for $0\leqslant k\leqslant n$ and $c\in W_n$.
(Proposition~\ref{distinct} shows that $C$ does 
indeed have cardinality $(k+1)$.)
In this situation
we call $c$ a \emph{lift} of the simplex $C$.  
\end{definition}

\begin{remark}
We chose the name ``lift'' in the previous definition
to emphasise the formal similarity with the concept of
the same name that appears in Definition~2.1 
of Wahl's paper~\cite{WahlAutomorphism}.
\end{remark}

A given simplex can have many lifts.
Choosing a lift for a
simplex induces an ordering of its vertices,
and all orderings occur in this way.
For if $c$ lifts a $k$-simplex
$C$ then so does $c s_{n-k+i+1}$, and
the induced orderings differ by transposition
of the $i$-th and $(i+1)$-st vertices
(see Proposition~\ref{permutations}).
This makes it simple to verify that $\C^n$ is
indeed a simplicial complex, for if $C$ is a
simplex of $\C^n$ and $D\subset C$ is a nonempty
subset, then we may choose a lift $c$
of $C$ such that $D$ is a terminal
segment in the induced ordering.
Then $c$ is also a lift of $D$.

The natural action of $W_n$ on $W_n/W_{n-1}$
extends to an action on $\C^n$.
For if $C$ is a simplex of $\C^n$ with
lift $c$, and if $w\in W_n$, then
$wC$ is a simplex of $\C^n$ with lift
$wc$.

We now give a concrete description of $\C^n$
for the families of Coxeter groups of type
$A_n$, $B_n$ and $D_n$ that were discussed in the introduction.  
Again, see section~6.7 of~\cite{Davis} for the concrete
descriptions of these groups.

\begin{example}[The family $A_n$]\label{An:example}
Let the sequence of diagrams $(\Gamma_n)_{n\geqslant 1}$ be
$(A_n)_{n\geqslant 1}$, so that $W_n=\Sigma_{n+1}$
is the symmetric group on $(n+1)$ letters.
Then $\C^n$ is the $n$-dimensional simplex $\Delta^n$ with the
action of $\Sigma_{n+1}$ that permutes the vertices.
For the vertex set of $\C^n$ is $\Sigma_{n+1}/\Sigma_n$,
which is isomorphic to $\{1,\ldots,n+1\}$ \emph{via} the map that
sends $\sigma\Sigma_{n}$ to $\sigma(n+1)$.
Under this isomorphism, an element $\sigma\in\Sigma_{n+1}$
is a lift of the $k$-simplex
\[
	C=\{\sigma(n-k+1),\ldots,\sigma(n+1)\},
\]
and every subset of $\{1,\ldots,n+1\}$ arises in this way.
\end{example}

\begin{example}[The family $B_n$]\label{Bn:example}
Let the sequence of diagrams $(\Gamma_n)_{n\geqslant 1}$ be
given by
$(B_{n+1})_{n\geqslant 1}$, so that $W_n=C_2\wr\Sigma_{n+1}$.
In this case $\C^n$ is isomorphic to the 
\emph{hyperoctahedron of dimension $n$}, which is
the simplicial complex whose vertex set is $\{\pm1,\ldots,\pm(n+1)\}$
and whose simplices are the subsets containing at most one
element from each pair $\{i,-i\}$.
In particular, its realisation is homeomorphic to the $n$-sphere.
The action of $C_2\wr\Sigma_{n+1}$ on the hyperoctahedron
is the one in which $\Sigma_{n+1}$ permutes the
pairs $\pm i$ while preserving their signs, and in which
the $i$-th copy of $C_2$ transposes $i$ and $-i$. 
To obtain this description, observe that 
the vertex set of $\C^n$ is 
$(C_2\wr\Sigma_{n+1})/(C_2\wr\Sigma_n)$,
which is isomorphic to $\{\pm 1,\ldots,\pm (n+1)\}$
\emph{via} the map that sends the coset of
$((\epsilon_1,\ldots,\epsilon_{n+1}),\sigma)$ to
$\epsilon_{n+1}\sigma(n+1)$.  (Here we are taking $C_2=\{\pm 1\}$.)
And under this isomorphism an element 
$((\epsilon_1,\ldots,\epsilon_{n+1}),\sigma)$ lifts the
$k$-simplex
\[
	C=\{\epsilon_{n-k+1}\sigma(n-k+1),\ldots,\epsilon_{n+1}\sigma(n+1)\},
\]
so that a subset of $\{\pm 1,\ldots,\pm(n+1)\}$ spans a simplex
of $\C^n$ if and only if it does not contain any element
and its negative.
\end{example}

\begin{example}[The family $D_n$]\label{Dn:example}
Let the sequence of diagrams $(\Gamma_n)_{n\geqslant 1}$ be
given by $(D_{n+2})_{n\geqslant 1}$, so that $W_n=H_{n+2}$ is the
kernel of the homomorphism $C_2\wr\Sigma_{n+2}\to C_2$
that takes the sum of the $C_2$-components.  In this case
$\C^n$ is the $n$-skeleton of the $(n+1)$-dimensional hyperoctahedron,
with the action inherited from the action of $C_2\wr\Sigma_{n+2}$.
(See Example~\ref{Bn:example}.)  In particular,
the realisation of $\C^n$ has the homotopy type of the
wedge of $(2^n -1)$ copies of the $n$-dimensional sphere.
To obtain this description observe that the vertex set
$H_{n+2}/H_{n+1}=(C_2\wr\Sigma_{n+2})/(C_2\wr\Sigma_{n+1})$
can be identified with $\{\pm1,\ldots,\pm(n+2)\}$ \emph{via}
the map sending the coset of 
$((\epsilon_1,\ldots,\epsilon_{n+2}),\sigma)$
to $\epsilon_{n+2}\sigma(n+2)$, 
and that under this identification the $k$-simplex with lift 
$((\epsilon_1,\ldots,\epsilon_{n+2}),\sigma)$ is
\[
	C=\{\epsilon_{n-k+2}\sigma(n-k+2),\ldots,\epsilon_{n+2}\sigma(n+2)\},
\]
so that a subset of $\{\pm 1,\ldots,\pm(n+2)\}$ spans
a simplex if and only if it does not contain any element
and its negative.
\end{example}

Recall from Definition~3.4 of~\cite{HatcherWahl} 
that a simplicial complex is called \emph{weakly Cohen-Macaulay
of dimension $n$} if it is $(n-1)$-connected
and the link of each $p$-simplex is
$(n-p-2)$-connected.
In each of the three examples above, $\C^n$ has
the homotopy type of a wedge of $n$-dimensional spheres,
and so is $(n-1)$-connected.
In fact, it is not hard to see that in these examples 
$\C^n$ is weakly Cohen-Macaulay
of dimension $n$.  This is an instance of the following
general fact.

\begin{theorem}\label{CohenMacaulay}
$\C^n$ is weakly Cohen-Macaulay of dimension $n$.
\end{theorem}

The proof of this theorem is assembled from Propositions~\ref{link},
\ref{homeo} and~\ref{basic-connectivity},
which are proved over the course of the next three sections.

\begin{proof}
By Proposition~\ref{link}, if $C$ is a $p$-simplex of $\C^n$
then $\lk_{\C^n}(C)\cong\C^{n-p-1}$.  It therefore suffices to
show that $\C^n$ is $(n-1)$-connected for all $n$,
or equivalently that the barycentric subdivision $\sd\C^n$
is $(n-1)$-connected for all $n$.  Now Proposition~\ref{homeo}
shows that $|\sd\C^n|$ is homeomorphic to the basic construction
$\U(W_n,|\Delta|)$, while Proposition~\ref{basic-connectivity}
shows that $\U(W_n,|\Delta|)$ is $(n-1)$-connected.
This completes the proof.
\end{proof}

\section{Links of simplices of $\C^n$}
\label{links:section}

\begin{proposition}\label{link}
Let $C$ be a $p$-simplex of $\C^n$.
Then $\lk_{\C^n}(C)\cong \C^{n-p-1}$.
\end{proposition}

\begin{proof}
Choose a lift $c$ of $C$.
Define 
\[
	\phi \colon W_{n-p-1}/W_{n-p-2}\longrightarrow W_n/W_{n-1}
\]
by $\phi( d W_{n-p-2}) = cd s_{n-p}\cdots s_n W_{n-1}$
for $d\in W_{n-p-1}$.
This is well defined since every 
generator of $W_{n-p-2}$ commutes
with $s_{n-p},\ldots,s_n$.  
Observe that the domain and range of $\phi$
are the vertex sets of $\C^{n-p-1}$ and $\C^n$ respectively.

\begin{claim}
$\phi$ is an injection.
\end{claim}

To prove this claim let $d,d'\in W_{n-p-1}$ satisfy
\[
	cd(s_{n-k}\cdots s_n) W_{n-1}=
	cd' s_{n-k}\cdots s_n W_{n-1}.
\]
Then
\[
	d^{-1}d'
	\in W_{n-p-1}\cap (s_{n-p}\cdots s_n) W_{n-1}(s_n\cdots s_{n-p})
	=
	W_{n-p-2},
\]
the latter equation by Proposition~\ref{intersection}.
Thus $d' W_{n-p-2}=d W_{n-p-2}$.

\begin{claim}
	$\phi$ sends simplices of $\C^{n-p-1}$ to
	simplices of $\lk_{\C^n}(C)$;
\end{claim}

To prove this, suppose 
that $D$ is an $i$-simplex of $\C^{n-p-1}$.
Let $d\in W_{n-p-1}$ be a lift of $D$.
Then
\[
	\phi D = 
	\{ cds_{n-p-i}\cdots s_n W_{n-1},
	\ldots, cds_{n-p}\cdots s_n W_{n-1}\}
\]
while
\begin{align*}
	C&= \{cs_{n-p+1}\cdots s_nW_{n-1},\ldots,cs_n W_{n-1},cW_{n-1}\}
	\\
	&= \{cds_{n-p+1}\cdots s_nW_{n-1},\ldots,cds_n W_{n-1},cdW_{n-1}\}.
\end{align*}
Thus $\phi D\cap C=\emptyset$ by Proposition~\ref{distinct}, and 
$\phi D \cup C$ is a simplex of $\C^n$ with lift $cd$,
so that $\phi D$ is a simplex of $\lk_{\C^n}(C)$ as claimed.

\begin{claim}
	Every simplex
	of $\lk_{\C^n}(C)$ has the form $\phi D$
	for some simplex $D$ of $\C^{n-p-1}$.
\end{claim}

To prove this, suppose that 
$\bar D$ is an $i$-simplex of $\lk_{\C^n}(C)$.
Then $\bar D\cap C = \emptyset$ and 
$\bar D \cup C$ is a simplex of $\C^n$.
Let $c'$ be a lift of $\bar D\cup C$,
and assume without loss that the ordering it
induces on $\bar D\cup C$
contains $\bar D$ as an initial segment and $C$ as a terminal
segment with the ordering induced by $c$.
Thus
\[
	\bar D
	=
	\{c'(s_{n-p-i}\cdots s_n) W_{n-1},
	\ldots, c'(s_{n-p}\cdots s_n)W_{n-1}\}
\]
and
\[
	c'(s_{n-p+j}\cdots s_n)W_{n-1}=c(s_{n-p+j}\cdots s_n)W_{n-1}
\]
for $j=1,\ldots,p+1$.
The latter gives $c^{-1}c'\in W_{n-p-1}$ by
Proposition~\ref{tuple-stabilizer},
so that $c'=cd$ for some $d\in W_{n-p-1}$.
Then $\bar D = \phi D$, where $D$ is the $i$-simplex of $\C^{n-p-1}$
with lift $d$.

We can now prove the proposition.
Combining the first claim with the third in the case of $0$-simplices,
we see that $\phi$ is an isomorphism between the vertex
sets of $\C^{n-p-1}$ and $\lk_{\C^n}(C)$.
The second and third claims then show that $\phi$ induces
an isomorphism of simplicial complexes from $\C^{n-p-1}$
to $\lk_{\C^n}(C)$.
\end{proof}

\section{The action of $W_n$ on $\C^n$ and $\sd\C^n$}
\label{action:section}

Recall that if $X$ is an abstract simplicial complex,
then its \emph{barycentric subdivision} $\sd X$ is the
abstract simplicial complex whose vertices are the
simplices of $X$, and in which $\{C_0,\ldots,C_k\}$
is a simplex if, after possibly reordering the elements,
we have $C_0\subset\cdots\subset C_k$.

Here we will study the barycentric subdivision $\sd\C^n$.
The action of $W_n$ on $\C^n$ induces an action of $W_n$
on $\sd\C^n$.  
This new action automatically has the property
that if $w\in W_n$ and $\mathbf{C}$ is a simplex of $\sd\C^n$,
then $w$ fixes every vertex of $\mathbf{C}\cap w\mathbf{C}$.
(See~p.115 of~\cite{Bredon}.)
In particular any element of the stabilizer of $\mathbf{C}$
must fix $\mathbf{C}$ pointwise.

\begin{lemma}
\label{transitivity-original}
The action of $W_n$ on $\C^n$ is transitive
on the set of $k$-simplices for each $k$.
For a $k$-simplex $C$ of $\C^n$,
any permutation of the vertices of $C$
is realised by some element of $W_n$ that fixes $C$.
Every simplex of $\C^n$ is a face of an
$n$-simplex.
\end{lemma}

\begin{proof}
For the first claim, let $C$ and $D$ be 
$k$-simplices of $\C^n$ with respective lifts
$c$ and $d$.  Then $(d c^{-1})C=D$.
For the second claim, let $C$ be a simplex
of $\C^n$ with lift $c$, and let $\phi$
be a permutation of the vertices of $C$.
Let $c'$ be the lift of $C$ whose
induced ordering is obtained from the ordering
induced by $c$ by applying $\phi$.
Then $c'c^{-1}$ sends each vertex
of $C$ to its image under $\phi$.
The final claim follows because a
$k$-simplex with lift $c$
is a face of the $n$-simplex with lift $c$.
\end{proof}

\begin{lemma}
\label{transitivity-subdivision}
$W_n$ acts transitively on the $n$-simplices of $\sd\C^n$.
Every simplex of $\sd\C^n$ is a face of an $n$-simplex.
\end{lemma}

\begin{proof}
Since $\C^n$ has dimension $n$,
the set of $n$-simplices of $\sd\C^n$
is in natural bijection with the set of $n$-simplices
of $\C^n$ equipped with an ordering of their vertices.
Lemma~\ref{transitivity-original} shows that the action
of $W_n$ on this set is transitive, proving the first claim.
The same lemma shows that every simplex of $\C^n$ is
a face of an $n$-simplex, and the second claim
follows immediately.
\end{proof}

To understand the action of $W_n$ on $\sd\C^n$
we may now concentrate on a single $n$-simplex.

\begin{definition}\label{Delta:definition}
Let $\Delta=\{a_0,\ldots,a_n\}$
denote the $n$-simplex of $\sd\C^n$ with vertices
$a_i=\{(s_{i+1}\cdots s_n)W_{n-1},\ldots,s_nW_{n-1},eW_{n-1}\}$
for $i=1,\ldots,n$.
For each $s\in S_n$ define a face $\Delta_s$ of $\Delta$
as follows.  If $s\in S_i\setminus S_{i-1}$ for $i=1,\ldots,n$ then
\[
	\Delta_{s}=\{a_0,\ldots,\widehat{a_i},\ldots,a_n\},
\]
and if $s\in S_{-1}$ then $\Delta_s=\Delta$.
\end{definition}

\begin{example}[The case $n=2$]
Let $n=2$.  The diagram on the left shows the
$2$-simplex $\{s_1s_2W_1,s_2W_1,W_1\}$ of $\C^2$ 
with lift the identity element, and the diagram
on the right shows $\Delta$ as a face of the subdivision
of this simplex.
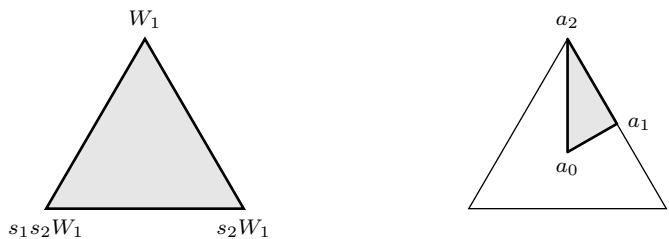
\[
\begin{tikzpicture}[scale=0.3, baseline=0]
	\path[draw, line width=1, fill=black!10!white]
		 (90:5) -- (210:5) -- (330:5) -- cycle;
	\node[above] at (90:5) {$\scriptstyle W_1$};
	\node[below] at (210:5) {$\scriptstyle s_1s_2W_1$};
	\node[below] at (330:5) {$\scriptstyle s_2 W_1$};
\end{tikzpicture}
\hspace{70 pt}
\begin{tikzpicture}[scale=0.3,baseline=0]
	\path[draw, line width=0.5]
		 (90:5) -- (210:5) -- (330:5) -- cycle;
	\path[draw, line width=1, line join=round, fill=black!10!white]
		 (90:0) -- (90:5) -- (210:-5/2) -- cycle;
	\node[above] at (90:5) {$\scriptstyle a_2$};
	\node[right] at (210:-5/2) {$\scriptstyle a_1$};
	\node[below] at (90:0) {$\scriptstyle a_0$};
\end{tikzpicture}
\]
The next diagrams show $\Delta_{s_2}$, $\Delta_{s_1}$,
$\Delta_t$  and $\Delta_u$
for $t\in S_0\setminus S_{-1}$, and $u\in S_{-1}$.
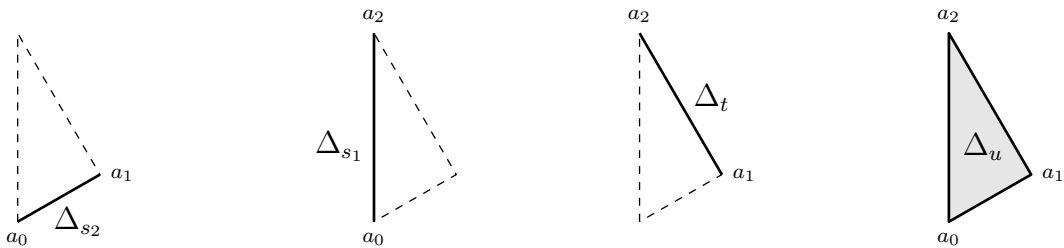
\[
\begin{tikzpicture}[scale=0.5,baseline=0]
	\path[draw, line width=0.5, line join=round,dashed]
		 (90:0) -- (90:5) -- (210:-5/2) -- cycle;
	\path[draw, line width=1, line join=round]
		 (210:-5/2) --(90:0);
	\node[right] at (210:-5/2) {$\scriptstyle a_1$};
	\node[below] at (90:0) {$\scriptstyle a_0$};
	\node at (1.6,0) {$\Delta_{s_2}$};
\end{tikzpicture}
\hspace{60 pt}
\begin{tikzpicture}[scale=0.5,baseline=0]
	\path[draw, line width=0.5, line join=round,dashed]
		 (90:0) -- (90:5) -- (210:-5/2) -- cycle;
	\path[draw, line width=1, line join=round]
		 (90:0) -- (90:5);
	\node[above] at (90:5) {$\scriptstyle a_2$};
	\node[below] at (90:0) {$\scriptstyle a_0$};
	\node at (-0.9,2) {$\Delta_{s_1}$};
\end{tikzpicture}
\hspace{60 pt}
\begin{tikzpicture}[scale=0.5,baseline=0]
	\path[draw, line width=0.5, line join=round,dashed]
		 (90:0) -- (90:5) -- (210:-5/2) -- cycle;
	\path[draw, line width=1, line join=round]
		 (90:5) -- (210:-5/2);
	\node[above] at (90:5) {$\scriptstyle a_2$};
	\node[right] at (210:-5/2) {$\scriptstyle a_1$};
	\node at (1.9,3.3) {$\Delta_{t}$};
\end{tikzpicture}
\hspace{60 pt}
\begin{tikzpicture}[scale=0.5,baseline=0]
	\path[draw, line width=1, line join=round, fill=black!10!white]
		 (90:0) -- (90:5) -- (210:-5/2) -- cycle;
	\node[above] at (90:5) {$\scriptstyle a_2$};
	\node[right] at (210:-5/2) {$\scriptstyle a_1$};
	\node[below] at (90:0) {$\scriptstyle a_0$};
	\node at (0.9,2) {$\Delta_{u}$};
\end{tikzpicture}
\]
Observe that $s_2$ acts on $\{s_1s_2W_1,s_2W_1,W_1\}$
as the reflection that fixes
$s_1s_2W_1$ and interchanges $W_1$ and $s_2W_1$,
and $\Delta_{s_2}$ is exactly the part of $\Delta$
fixed by this reflection. 
Similarly, $s_1$ acts as the reflection that fixes
$W_1$ and interchanges $s_1s_2W_1$ and $s_2W_1$, and
$\Delta_{s_1}$ is precisely the part of $\Delta$
fixed by this reflection.
Next, $t$ can be thought of as 
``reflection in the edge $\{W_1,s_2W_1\}$'',
and $\Delta_t$ is precisely the part of $\Delta$ that it fixes.
Finally, $u$ acts trivially on $\Delta$ since it lies
in $W_1$ and commutes with $s_1$ and $s_2$, and indeed
$\Delta_u=\Delta$ is the part of $\Delta$ fixed by $u$.
\end{example}

\begin{lemma}
\label{stabilizer}
Let $F$ be a face of $\Delta$.
Then the stabilizer of $F$ under the 
action of $W_n$ is the subgroup generated 
by those $s\in S_n$ for which $F\subset \Delta_s$.
\end{lemma}

\begin{proof}
For the purposes of this proof, given $i\geqslant 0$
we write $S_{=i}$ for the difference $S_i\setminus S_{i-1}$.
So for $i\geqslant 1$ we have $S_{=i}=\{s_i\}$, while $S_{=0}$
is the set of elements of $S_0$ that do not commute with $s_1$.

For $i\geqslant 0$, the stabilizer of $a_i$ is the subgroup
of $W_n$ generated by $S_n\setminus S_{=i}$.  
For Proposition~\ref{permutations} shows that
$s_{i+1},\ldots,s_n$ fix 
$a_i=\{s_{i+1}\cdots s_nW_{n-1},\ldots,s_nW_{n-1},W_{n-1}\}$
and generate all permutations of its elements,
while Proposition~\ref{tuple-stabilizer} shows 
that the subgroup consisting of elements that fix
every element of $a_i$ is $W_{i-1}$.

Let $F=\{a_{i_1},\ldots,a_{i_r}\}$.
Then the stabilizer of $F$ is the intersection
of the stabilizers of the $a_{i_j}$.  By the last
paragraph this is the intersection of the subgroups
generated by the sets $S_n\setminus S_{=i_j}$,
and by a general result (see Theorem~4.1.6 of~\cite{Davis}) 
this is the subgroup generated by 
$\bigcap (S_n\setminus S_{=i_j}) = S_n\setminus\bigcup S_{=i_j}$.

Now $F\subset \Delta_s$ for all $s\in S_{-1}$, and $F\subset \Delta_s$
for $s\in S_{=i}$ if and only if $a_i\not\in F$.
Thus the set of $s$ such that $F\subset\Delta_s$
is $S\setminus\bigcup S_{=i_j}$.
Thus the subgroup generated by the $s$ such
that $F\subset\Delta_s$ is precisely the
pointwise stabilizer of $F$.
\end{proof}

\section{The barycentric subdivision of $\C^n$ and the basic construction}
\label{basic:section}

The `basic construction' is a method for building
and studying certain spaces with group action.
It can be used, for example, to study the topology
of the Coxeter complex and Davis complex of a Coxeter group.
In this section we will show that $|\sd\C^n|$ is an
instance of the basic construction, 
and we will use this to show that it is $(n-1)$-connected.

To begin we recall the relevant notions from 
section 5.1 of \cite{Davis}.
(These are tailored to the case of Coxeter groups. For an 
approach to the basic construction that applies to more
general groups see chapter~II.12 of~\cite{BridsonHaefliger}.)
Let $(W,S)$ be a Coxeter system.
A \emph{mirrored space} over $S$ is a space $X$
together with subspaces $X_s\subset X$, called \emph{mirrors},
one for each $s\in S$.
We assume that $X$ is a CW-complex and that the mirrors are subcomplexes.
The \emph{basic construction} is the space
\[
	\U(W,X) = (W\times X)/\sim
\]
where $(v,x)\sim (w,y)$ if and only if $x=y$ and $v^{-1}w$
belongs to the subgroup generated by the $s\in S$
for which $x\in X_s$.  The basic construction is equipped with the action
of $W$ by left translation, and we identify $X$ with
the image of $\{e\}\times X$ in $\U(W,X)$.
Observe that $\U(W,X)$ has the structure of a CW-complex
in which each translate $wX$ is a subcomplex.

For us the most important feature of the basic
construction is that it can be described as an increasing
union of chambers, meaning copies of $X$, 
as we now recall from section~8.1 of~\cite{Davis}.  
Given $w\in W$, let $\In(w)=\{s\in S\mid \ell(ws)<\ell(w)\}$
denote the set of letters with which a reduced expression
for $w$ can end, and let $X^{\In(w)}=\bigcup_{s\in\In(w)}X_s$
denote the corresponding union of mirrors.
Now order the elements of $W$ as $w_0,w_1,w_2,\ldots$
where $w_0=e$ and $\ell(w_{m})\leqslant \ell(w_{m+1})$ 
for $m\geqslant 0$. Define
\[
	P_m=\bigcup_{i=0}^m w_iX,
\]
so that $\U(W,X)$ is the increasing union of the subcomplexes $P_m$.
Then
\[
	P_m=P_{m-1}\cup w_mX
	\qquad\text{and}\qquad
	P_{m-1}\cap w_mX = w_m X^{\In(w_m)}.
\]
The latter equation is by Lemma~8.1.1 of~\cite{Davis},
and it will be useful to us since it specifies exactly
how each chamber is attached to its predecessory.

Now let us show that $|\sd\C^n|$ is an instance of the
basic construction.  Fix the Coxeter system $(W_n,S_n)$.
Recall the simplex $\Delta$ and the faces $\Delta_s$
introduced in Definition~\ref{Delta:definition}.
We make $|\Delta|$ into a mirrored space over $S_n$
by defining the mirror $|\Delta|_s$ to be
the subspace $|\Delta_s|$ of $|\Delta|$. 
We may therefore form the basic construction $\U(W_n,|\Delta|)$.
The inclusion $|\Delta|\hookrightarrow |\sd\C^n|$
sends $|\Delta|_s$ into the subset fixed by $s$,
and so extends uniquely to a $W_n$-equivariant map
$\U(W_n,|\Delta|)\to|\sd\C^n|$.

\begin{proposition}
\label{homeo}
The map $\U(W_n,|\Delta|)\to|\sd\C^n|$ is a
homeomorphism.
\end{proposition}

\begin{proof}
The map is surjective because any point of $|\sd\C^n|$
is in a translate of $|\Delta|$.  This follows from
Lemma~\ref{transitivity-subdivision}, which shows
that every simplex of $\sd\C^n$ is a face
of a translate of $\Delta$.

The map is injective because if $x\in |\Delta|$,
then the image of $x$ in $\U(W_n,|\Delta|)$ 
has stabilizer generated by those $s$ for which $x\in|\Delta|_s$.
To see this, let $F$ denote the unique face 
of $\Delta$ for which
$x$ lies in the interior of $|F|$.
Then $x\in|\Delta|_s=|\Delta_s|$ if and only
if $F\subset\Delta_s$, and the stabilizer of $x$
is precisely the stabilizer of $F$.
The result then follows from
Lemma~\ref{stabilizer}.

The map is a homeomorphism because $|\sd\C^n|$
has the weak topology with respect to the realizations
of its simplices.  By Lemma~\ref{transitivity-subdivision}
this coincides with the weak topology with respect
to the realizations of its $n$-simplices.
This is exactly the topology of $\U(W_n,|\Delta|)$.
\end{proof}

\begin{proposition}\label{basic-connectivity}
$\U(W_n,|\Delta|)$ is $(n-1)$-connected.
\end{proposition}

This relies on the following two lemmas.

\begin{lemma}\label{in-connectivity}
For $w\in W_n$, $w\neq e$, the space
$|\Delta|^{\In(w)}$ is $(n-2)$-connected.
\end{lemma}

\begin{proof}
The set $\In(w)$ is nonempty since $w\neq e$.
Thus $|\Delta|^{\In(w)}$ is either $|\Delta|$,
or it is a nonempty union of facets of $|\Delta|$.
In the first case it is contractible,
and in the second case it is either contractible (if not all facets
are in the union) or it is $\partial|\Delta|\cong S^{n-1}$
(if all facets are in the union).
In all cases it is $(n-2)$-connected.
\end{proof}

\begin{lemma}\label{excision}
Let $n\geqslant 1$.
Suppose that $(X;A,B)$ is a CW-triad in which
$A$ and $B$ are $(n-1)$-connected and $C=A\cap B$
is $(n-2)$-connected.  Then $X$ is $(n-1)$-connected.
\end{lemma}

\begin{proof}
For $n=1$ this is immediate since the union
of two path-connected spaces with nonempty intersection
is path-connected.  So we assume that $n\geqslant 2$.
The pairs $(A,C)$ and $(B,C)$ are $(n-1)$-connected,
and $C$ is path-connected,
so that Theorem~4.23 of~\cite{Hatcher} can be applied to  
show that 
$\pi_i(A,C)\to\pi_i(X,B)$ is an isomorphism
for $i<2n-2$, and in particular for $i\leqslant (n-1)$.
Thus $(X,B)$ is $(n-1)$-connected,
and the same then follows for $X$ itself.
\end{proof}

\begin{proof}[Proof of Proposition~\ref{basic-connectivity}]
If $n=0$ then the claim is that $\U(W_n,|\Delta|)$
is nonempty, which holds vacuously.
So we may assume that $n\geqslant 1$.

As in the discussion at the start of the section,
order the elements of $W_n$ as $w_0,w_1,w_2,\ldots$
starting with the identity and respecting the length.
Then $\U(W_n,|\Delta|)$ is the union of subcomplexes
$P_0\subset P_1\subset P_2\subset\cdots$
where $P_0=|\Delta|$ and
\[
	P_m=P_{m-1}\cup w_m |\Delta|
	\qquad\text{with}\qquad
	P_{m-1}\cap w_m|\Delta| = w_m|\Delta|^{\In(w_m)}.
\]
It will suffice to show that each $P_m$ is $(n-1)$-connected.  
We do this by induction on $m$.

In the initial case $m=0$ we have $P_0=e|\Delta|$,
which is contractible and so the claim holds.
For the induction step we take $m\geqslant 1$
and assume that $P_{m-1}$ is $(n-1)$-connected.
Then $P_m=P_{m-1}\cup w_m|\Delta|$ 
is the union of the subcomplexes $P_{m-1}$ and $w_m|\Delta|$,
and their intersection $w_m|\Delta|^{\In(w_m)}$
is $(n-2)$-connected by Lemma~\ref{in-connectivity}.
Thus $(P_m;P_{m-1},w_m|\Delta|)$
is a CW-triad in which the subspaces
$P_{m-1}$ and $w_m|\Delta|$ are $(n-1)$-connected
and their intersection  
is $(n-2)$-connected.
It now follows from Lemma~\ref{excision}
that $P_m$ is $(n-1)$-connected as required.
\end{proof}

\section{The ordered simplices of $\C^n$}
\label{ordered:section}
In this section we introduce a semisimplicial set
$\D^n$ and identify it as the semisimplicial set
of ordered simplices in $\C^n$.  We then use the
fact that $\C^n$ is weakly Cohen-Macaulay of dimension $n$
to deduce that the geometric realisation $\|\D^n\|$ is
$(n-1)$-connected, an approach we learned from
Wahl's paper~\cite{WahlAutomorphism} (see in particular
Proposition~7.9 of~\cite{WahlAutomorphism}, 
which is due to Randal-Williams).
In this section and the next we will use semisimplicial spaces
and their realisations.  The background material we require
can be found in section~2 of~\cite{RW}.

\begin{definition}
Let $\D^n$ denote the semisimplicial set 
with $k$-simplices.
\[
	\D^n_k = \left\{
	\begin{array}{ll}
		W_n/W_{n-k-1}	& k\leqslant n
		\\
		\emptyset & k>n
	\end{array}
	\right.
\]
and with face maps
\[
	d_i\colon W_n/W_{n-k-1}
	\longrightarrow
	W_n/W_{n-k}
\]
defined by
\[
	d_i(cW_{n-k-1})
	=
	c(s_{n-k+i}\cdots s_{n-k+1})W_{n-k}
\]
for $i=0,\ldots,k$.
\end{definition}

It is a simple exercise to verify that the face maps
$d_i$ satisfy the relations $d_i\circ d_j = d_{j+1}\circ d_{i}$
for $i<j$.  Alternatively, it is a consequence
of the proof of Proposition~\ref{CnDn} below.

\begin{definition}
Let $X$ be a simplicial complex.
By an \emph{ordered simplex of $X$},
we mean a simplex of $X$ equipped with an
ordering of its vertices.
The \emph{semi-simplicial set 
of ordered simplices in $X$}, denoted $X^\ord$,
has for its $k$-simplices the ordered $k$-simplices
in $X$, with face maps $d_i$ given by forgetting
the $i$-th vertex of an ordered simplex.
\end{definition}

\begin{proposition}\label{CnDn}
$\D^n$ is isomorphic to $\C^{n,\ord}$.
\end{proposition}

\begin{proof}
We define $\phi_k\colon \D^n_k\to \C^{n,\ord}_k$ by
\[
	\phi_k(cW_{n-k-1})
	=
	\{
		c (s_{n-k+1}\cdots s_n) W_{n-1},\ 
		\ldots,\ 
		c s_nW_{n-1},\ 
		c W_{n-1}	
	\}
\]
for $cW_{n-k-1}\in W_n/W_{n-k-1}$.
In other words, $\phi_k(cW_{n-k-1})$ is the $k$-simplex with lift $c$,
equipped with the ordering induced by $c$.
The map $\phi_k$ is well defined because the generators of $W_{n-k-1}$ all
commute with $s_{n-k+1},\ldots,s_n$.
It is surjective because by definition every simplex admits a lift,
and any ordering of a simplex is afforded by some lift
(see the paragraph following Definition~\ref{Cndefinition}).
It is injective because if $\phi_k(cW_{n-k-1})=\phi_k(c'W_{n-k-1})$
then $cs_i\cdots s_nW_{n-1}=c's_i\cdots s_nW_{n-1}$ for $i=n-k+1,\ldots,n+1$,
so that $cW_{n-k-1}=c'W_{n-k-1}$ by Proposition~\ref{tuple-stabilizer}.

To complete the proof we must show that the face maps
in $\C^{n,\ord}$ and $\D^n$ are compatible under the $\phi_k$.
In other words, given $0\leqslant i\leqslant k\leqslant n$,
we must show that  
\[
	\phi_{k-1}\circ d_i = d_i\circ\phi_k.
\]
Observe from the definition of $d_i$ in $\D^n$
that for $i\geqslant 1$ we have
$d_i(cW_{n-k-1}) = d_{i-1}(cs_{n-k+i}W_{n-k+1})$.
On the other hand, Proposition~\ref{permutations}
shows that $\phi_k(cW_{n-k-1})$ and $\phi_k(cs_{n-k+i}W_{n-k-1})$
differ only by the transposition of their $(i-1)$-st and $i$-th
vertices, so that 
$d_i(\phi_k(cW_{n-k-1}))=d_{i-1}(\phi_k(cs_{n-k+i}W_{n-k-1}))$.
Thus the claim will follow by induction on $i$
so long as we can show that
\[
	\phi_{k-1}\circ d_0 = d_0\circ\phi_k.
\]
This follows by inspection.
\end{proof}

\begin{corollary}\label{Dn}
$\|\D^n\|$ is $(n-1)$-connected.
\end{corollary}

\begin{proof}
Theorem~\ref{CohenMacaulay} shows that $\C^n$ is weakly
Cohen-Macaulay of dimension $n$.
Proposition~7.9 of~\cite{WahlAutomorphism} shows that
if a simplicial complex $X$ is weakly Cohen-Macaulay
of dimension $n$, then $\|X^\ord\|$ is $(n-1)$-connected.
Consequently $\|\C^{n,\ord}\|$ is $(n-1)$-connected,
and by Proposition~\ref{CnDn} the same holds for $\|\D^n\|$.
\end{proof}

\section{Completion of the proof}
\label{completion:section}

We now complete the proof of the main theorem.
This section is modelled closely on section~5
of~\cite{RW}, from which there is little essential difference.
It is also similar to the proof of Theorem~2 of~\cite{Kerz}.

We regard $\D^n$ as a simplicial space by equipping 
its constituent sets with the discrete topology.
Then we form a semisimplicial space
\[
	EW_n\times_{W_n}\D^n
\]
by setting
$(EW_n\times_{W_n}\D^n)_k = EW_n\times_{W_n}\D^n_k$
and using the face maps obtained from those of $\D^n$.

\begin{lemma}\label{augmentation}
The projection $EW_n\times_{W_n}\D^n_0\to BW_n$
makes $EW_n\times_{W_n}\D^n$ into an augmented simplicial
space over $BW_n$, and the induced map
$\|EW_n\times_{W_n}\D^n\|\to BW_n$ is $(n-1)$-connected.
\end{lemma}

\begin{proof}
The composites of the projection with 
$d_0$ and $d_1$ coincide, so that the projection is indeed
an augmentation.
Since $EW_n\to BW_n$ is a locally trivial principal $W_n$-bundle,
it follows that $\|EW_n\times_{W_n}\D^n\|\to BW_n$ is a locally
trivial bundle with fibre $\|W_n\times_{W_n}\D^n\|\cong\|\D^n\|$,
which is $(n-1)$-connected by Corollary~\ref{Dn},
so that the map itself is $(n-1)$-connected.
\end{proof}

\begin{lemma}\label{columns}
There are homotopy equivalences 
$EW_n\times_{W_n}\D^n_k \simeq BW_{n-k-1}$
under which the face maps 
$d_i\colon EW_n\times_{W_n}\D^n_k
\longrightarrow EW_n\times_{W_n}\D^n_{k-1}$
are all homotopic to the stabilization map
$BW_{n-k-1}\to BW_{n-k}$, and under which
the composite
$EW_n\times_{W_n}\D^n_0\to\|EW_n\times_{W_n}\D^n\|\to BW_n$
is homotopic to the stabilization map
$BW_{n-1}\to BW_n$.
\end{lemma}

\begin{proof}
There is an isomorphism
\[
	EW_n\times_{W_n}\D^n_k
	=
	EW_n\times_{W_n}(W_n/W_{n-k-1})
	\xrightarrow{\ \cong\ }
	EW_n/W_{n-k-1}
\]
sending the orbit of $(x,cW_{n-k-1})$
to the orbit of $c^{-1}x$.
This identifies $d_i$ with the map
\[
	EW_n/W_{n-k-1}\longrightarrow EW_n/W_{n-k}
\]
sending the $W_{n-k-1}$-orbit of $x$ 
to the $W_{n-k}$-orbit of $(s_{n-k+1}\cdots s_{n-k+i})x$.
Since $(s_{n-k+1}\cdots s_{n-k+i})$ commutes with every
element of $W_{n-k-1}$, this map is homotopic to the
one sending the $W_{n-k-1}$-orbit of $x$ to the 
$W_{n-k}$-orbit of $x$.
Now the equivariant homotopy equivalences
\[
	EW_{n-k-1}\to EW_n,
	\qquad\qquad
	EW_{n-k}\to EW_n
\]
induce homotopy equivalences
\[
	BW_{n-k-1}\to EW_n/W_{n-k-1},
	\qquad\qquad
	BW_{n-k}\to EW_n/W_{n-k}
\]
under which the map $EW_n/W_{n-k-1}\to EW_n/W_{n-k}$
just described becomes the stabilization map.
\end{proof}

The skeletal filtration of $\|EW_n\times_{W_n}\D^n\|$
leads to a first-quadrant spectral sequence
\[
	E^1_{k,l}=H_l(EW_n\times_{W_n}\D^n_k)
	\Longrightarrow
	H_{k+l}(\|EW_n\times_{W_n}D^n\|)
\]
in which the differential $d^1$ is given by the alternating
sum $\sum_{i=0}^k(-1)^i(d_i)_\ast$ of the maps induced
by the face maps.
Lemma~\ref{columns} allows us to identify the $E^1$-term
of this spectral sequence:
$E^1_{k,l}=H_l(BW_{n-k-1})$, 
and $d^1\colon E^1_{k,l}\to E^1_{k-1,l}$
is the stabilization map $H_l(BW_{n-k-1})\to H_l(BW_{n-k})$
if $k$ is even, and is zero if $k$ is odd.

\begin{lemma}\label{ss}
Assume that for all $m<n$ the stabilization map
$
	H_l(BW_{m-1})\to H_l(BW_m)
$
is an isomorphism in degrees $2l\leqslant m$.
Then the spectral sequence has the following properties:
\begin{enumerate}
	\item
	$E^\infty_{0,l}=\cdots = E^2_{0,l}=E^1_{0,l}$ for $2l\leqslant n$.
	\item
	$E^\infty_{k,l}=0$ for $k>0$ and $2(k+l)\leqslant n$.
\end{enumerate}
\end{lemma}

\begin{proof}
The assumption allows us to deduce that  
$E^2_{k,l}=0$ when $k\geqslant 1$ is odd and $2l+k+1\leqslant n$,
and that
$E^2_{k,l}=0$ when $k\geqslant 2$ is even and $2l+k\leqslant n$.
For in the first case $d^1\colon E^1_{k+1,l}\to E^1_{k,l}$
is the stabilization map $H_l(BW_{n-k-2})\to H_l(BW_{n-k-1})$,
and in the second case $d^1\colon E^1_{k,l}\to E^1_{k-1,l}$
is the stabilization map $H_l(BW_{n-k-1})\to H_l(BW_{n-k})$,
and our assumption means that both are isomorphisms in the given range.

To prove the first claim, observe that 
$E^2_{0,l}=E^1_{0,l}$ since $d^1\colon E^1_{1,l}\to E^1_{0,l}$
is zero.  The remaining differentials with target in
bidegree $(0,l)$ are $d^k\colon E^k_{k,l-k+1}\to E^k_{0,l}$
with $k\geqslant 2$, and these have domain zero since
$2(l-k+1)+k\leqslant 2(l-k+1)+k+1=2l-k+2\leqslant 2l\leqslant n$
so that $E^2_{k,l-k+1}=0$.
To prove the second claim, 
observe that if $2(k+l)\leqslant n$ and $k>0$,
then certainly $2l+k< 2l+k+1\leqslant 2(l+k)\leqslant n$,
so that $E^2_{k,l}=0$.
\end{proof}

We can now complete the proof of the main theorem,
showing by induction on $n\geqslant 2$ that
$H_l(BW_{n-1})\to H_l(BW_n)$ is an isomorphism
for $2l\leqslant n$.
For $n=2$ the claim is that $H_l(BW_{1})\to H_l(BW_2)$ is an
isomorphism for $l=0,1$.  For $l=0$ this is trivial 
since both $BW_1$ and $BW_2$ are connected.
For $l=1$ this follows from the well-known fact that
if $(W,S)$ is a Coxeter system then $H_1(BW)$ is isomorphic
to the elementary abelian $2$-group generated by the
elements of $S$, subject to the relation that identifies 
$s,t\in S$ if $m_{st}$ is odd.
Take $n>2$ and suppose that the theorem holds for all smaller
integers.  
Lemma~\ref{columns} shows that the composite
\[
	H_{l}(W_{n-1})=E^1_{0,l}
	\to
	E^\infty_{0,l}
	\to
	H_l(\|EW_n\times_{W_n}\D^n\|)
	\to
	H_l(BW_n)
\]
is the stabilization map, and we must show that this
is an isomorphism for $2l\leqslant n$.
Lemma~\ref{ss} shows that
the first two arrows are isomorphisms in this range, 
while Lemma~\ref{augmentation} shows that the last map
is an isomorphism for $l\leqslant n-1$, which 
holds since $2l\leqslant n$ and $n\geqslant 3$.

\bibliographystyle{plain}
\bibliography{coxeter}
\end{document}